\documentclass[12pt,twoside]{amsart}

\usepackage{amsmath,amsthm,amscd,amssymb,mathrsfs,graphicx,amsfonts,mathrsfs}
\usepackage{amsfonts}
\usepackage{amssymb,enumerate}
\usepackage{amsthm}
\usepackage[all]{xy}
\usepackage{hyperref}
\allowdisplaybreaks[4] \footskip=12pt
\renewcommand{\uppercasenonmath}[1]{}

\numberwithin{equation}{section} \theoremstyle{plain}
\newtheorem*{thm*}{Main Theorem}
\newtheorem{thm}{Theorem}[section]
\newtheorem{cor}[thm]{Corollary}
\newtheorem*{cor*}{Corollary}
\newtheorem{lem}[thm]{Lemma}
\newtheorem*{lem*}{Lemma}

\newtheorem*{fact*}{Fact}

\newtheorem*{nota*}{Notation}
\newtheorem{prop}[thm]{Proposition}
\newtheorem*{prop*}{Proposition}
\newtheorem{rem}[thm]{Remark}
\newtheorem*{rem*}{Remark}

\newtheorem*{observation*}{Observation}

\newtheorem*{exa*}{Example}
\newtheorem{df}[thm]{Definition}
\newtheorem*{df*}{Definition}

\newtheorem*{con*}{Construction}

\renewcommand{\geq}{\geqslant}
\renewcommand{\leq}{\leqslant}

\begin{document}
\begin{center}
{\large  \bf  Cosupport for triangulated categories}

\vspace{0.5cm} Xiaoyan Yang\\
Department of Mathematics, Northwest Normal University, Lanzhou 730070,
China
E-mail: yangxy@nwnu.edu.cn
\end{center}

\bigskip
\centerline { \bf  Abstract}
\leftskip10truemm \rightskip10truemm \noindent The goal of the article is to better understand cosupport in triangulated categories since it is still quite mysterious. We study boundedness of local cohomology and local homology functors using Koszul objects, give some characterizations of cosupport and get some results that, in special
cases, recover and generalize the known results about the usual cosupport. Also we include some computations of cosupport, settle the comparison of support and cosupport of cohomologically finite objects.  Finally, we assign to any object of the
category a subset of $\mathrm{Spec}R$, called the big cosupport.\\
\vbox to 0.3cm{}\\
{\it Keywords:} Koszul object; cosupport; support\\
{\it 2010 Mathematics Subject Classification:} 13D45; 18E35; 18E30.

\leftskip0truemm \rightskip0truemm
\bigskip
\section* { \bf Introduction}
Over the last few decades, the theory of support varieties has played an increasing important role in various aspects of representation theory as this theory
has made it possible to apply methods of algebraic geometry to obtain representation theoretic information. The prototype for this has been Quillen$'$s \cite{Q} description of the algebraic variety corresponding to the cohomology ring of a finite group, based on which Carlson \cite{C} introduced support varieties for modular representations.
 Their work has inspired the development of analogous theories in various
contexts:
restricted Lie algebras \cite{FP}, complete intersections in commutative algebra \cite{A,AB}, certain finite dimensional algebras \cite{EHS}, finite group schemes \cite{FP1,FP2}.

Based on certain localization functors on the category, Benson, Iyengar and Krause \cite{BIK,BIK1,BIK2} developed theories of support and cosupport for objects in any compactly generated triangulated categories admitting set-indexed coproducts. The foundation of their approach is
 constructions of local cohomology and local homology functors on triangulated categories with respect to a central ring of operators. These works have influenced some of their subsequent work on this topic: Avramov
and Iyengar \cite{AI} address the problem of realizing modules over arbitrary associative rings with
prescribed cohomological support; Krause \cite{HK} studies the classification of thick subcategories of modules over commutative noetherian rings. Lastly, their works play a pivotal role on a classification theorem for the localizing subcategories of the stable module category of a finite group (see \cite{BIK3}).

Despite the many ways in which cosupport is dual to the more established notion of support, cosupport seems to be  more elusive, even in the setting of commutative noetherian rings. Let $R$ be a commutative noetherian graded ring and $\mathscr{T}$ a compactly generated
$R$-linear triangulated category with small coproducts. This paper focuses on cosupport in $\mathscr{T}$.

Given a point for $\mathfrak{p}\in\mathrm{Spec}R$. For
each compact object $C$ in $\mathscr{T}$ and each injective $R$-module $I$, Brown representability
yields an object $T_{C/\hspace{-0.1cm}/\mathfrak{p}}(I)$ in $\mathscr{T}$ such that
\begin{center}$\mathrm{Hom}^\ast_\mathscr{T}(-,T_{C/\hspace{-0.1cm}/\mathfrak{p}}(I))=\mathrm{Hom}^\ast_R(\mathrm{H}^\ast_{C/\hspace{-0.1cm}/\mathfrak{p}}(-),I)$.\end{center}
In Section 3, we characterize cosupport of objects using the boundedness of local cohomology and local homology, seek an efficient way to compute cosupport.

\vspace{2mm} \noindent{\bf Theorem A.}\label{Th1.4} {\it{For any object $X$ in $\mathscr{T}$, one has that \begin{center}$\mathrm{cosupp}_{R}X=\bigcup_{C\in\mathscr{T}^c,\ \mathfrak{p}\in\mathrm{Spec}R}\mathrm{cosupp}_R\mathrm{Hom}^\ast_\mathscr{T}(V^{\mathcal{Z}(\mathfrak{p})}X,T_{C/\hspace{-0.1cm}/\mathfrak{p}}(I(\mathfrak{p})))$,\end{center}where $I(\mathfrak{p})$ is the injective envelope of $R/\mathfrak{p}$ for $\mathfrak{p}\in\mathrm{Spec}R$.}}

\vspace{2mm}
In Section 4, we settle the comparison of support and cosupport, obtain the following inclusion for cohomologically finite objects.

\vspace{2mm} \noindent{\bf Theorem B.}\label{Th1.4} {\it{Let $X$ be an cohomologically finite object in $\mathscr{T}$. If $\mathscr{T}$ is generated by a single compact object, or $\mathrm{dim}(\mathrm{supp}_{R}X)<\infty$, then there is an inequality
 \begin{center}$\mathrm{cosupp}_{R}X\subseteq\mathrm{supp}_{R}X$.\end{center}}}

\vspace{2mm}
In Section 5, the big cosupport of objects are studied, and the relation of this notion to cosupport are found.

Our results show that one can get a satisfactory version of the cosupport theories
in the setting of triangulated categories, compatible with the known results for the cosupport for modules and complexes.

\bigskip
\section{\bf Preliminaries}
This section collects some notions of triangulated categories
for use throughout this paper.
For terminology we shall follow \cite{BIK}, \cite{BIK1}, \cite{BIK2} and \cite{ASS}.

\vspace{2mm}
{\bf Compact generation.} Let $\mathscr{T}$ be a triangulated category admitting set-indexed coproducts. An object $C$ in $\mathscr{T}$ is compact if the functor $\mathrm{Hom}_\mathscr{T}(C,-)$ commutes with all coproducts;
we write $\mathscr{T}^c$ for the full subcategory of compact objects in $\mathscr{T}$. The category $\mathscr{T}$ is compactly
generated if it is generated by a set of compact objects. A localizing subcategory of
$\mathscr{T}$ is a full triangulated subcategory that is closed under taking coproducts. We write $\mathrm{Loc}_\mathscr{T}(\mathscr{X})$
for the smallest localizing subcategory containing a given class of objects $\mathscr{X}$ in $\mathscr{T}$, and call it
the localizing subcategory generated by $\mathscr{X}$.
 Analogously, a colocalizing subcategory of $\mathscr{T}$ is a full triangulated
subcategory that is closed under taking all products, and $\mathrm{Coloc}_\mathscr{T}(\mathscr{X})$ denotes the
colocalizing subcategory of $\mathscr{T}$ that is cogenerated by $\mathscr{X}$. We write $\mathrm{Thick}_\mathscr{T}(\mathscr{X})$ for the smallest thick
subcategory containing $\mathscr{X}$.

Recall that we write $\Sigma$ for the suspension on $\mathscr{T}$. For objects $X$ and $Y$ in $\mathscr{T}$, let
\begin{center}$\mathrm{Hom}^\ast_\mathscr{T}(X,Y)=\bigoplus_{i\in\mathbb{Z}}\mathrm{Hom}_\mathscr{T}(X,\Sigma^iY)$\end{center}
be the graded abelian group of morphisms. Set $\mathrm{End}^\ast_\mathscr{T}(X)=\mathrm{Hom}^\ast_\mathscr{T}(X,X)$, this is a
graded ring, and $\mathrm{Hom}^\ast_\mathscr{T}(X,Y)$ is a right $\mathrm{End}^\ast_\mathscr{T}(X)$ and left $\mathrm{End}^\ast_\mathscr{T}(Y)$-bimodule.

\vspace{2mm}
{\bf Central ring actions.}
Let $R$ be a commutative graded ring, i.e., $R$ is $\mathbb{Z}$-graded and satisfies $r\cdot s=(-1)^{|r||s|}s\cdot r$
for any homogeneous elements $r,s$ in $R$, where $|r|$ denote the degree of $r$. We say that the triangulated category $\mathscr{T}$ is
$R$-linear if there is a homomorphism $R\rightarrow Z^\ast(\mathscr{T})$ of graded rings, where
$Z^\ast(\mathscr{T})$ is the graded center of $\mathscr{T}$.  This yields for each object $X$ a homomorphism $\phi_X:R\rightarrow\mathrm{End}^\ast_\mathscr{T}(X)$
of graded
rings such that for all objects $X, Y\in\mathscr{T}$ the $R$-module structures on $\mathrm{Hom}^\ast_\mathscr{T}(X,Y)$
induced by $\phi_X$ and $\phi_Y$ agree, up to the usual sign rule.

We write $\mathrm{Spec}R$ for the set of homogeneous
prime ideals of $R$.
Fix a point $\mathfrak{p}\in\mathrm{Spec}R$. We write $R_\mathfrak{p}$ for the homogeneous localization of $R$ with respect to
$\mathfrak{p}$; it is a graded local ring in the sense of Bruns and Herzog \cite[1.5.13]{BH}, with maximal ideal
$\mathfrak{p}R_\mathfrak{p}$. Given an $R$-module $M$ we let $M_\mathfrak{p}$ denote the homogeneous localization of
$M$ at $\mathfrak{p}$. $M$ is called $\mathfrak{p}$-local if the
natural map $M\rightarrow M_\mathfrak{p}$ is bijective. $M$ is called $\mathfrak{p}$-torsion if each element of $M$ is annihilated by a power
of $\mathfrak{p}$. We also set \begin{center}$\mathcal{U}(\mathfrak{p})=\{\mathfrak{q}\in \mathrm{Spec}R\hspace{0.03cm}|\hspace{0.03cm} \mathfrak{q}\subseteq\mathfrak{p}\}$ and $\mathcal{Z}(\mathfrak{p})=\{\mathfrak{q}\in \mathrm{Spec}R\hspace{0.03cm}|\hspace{0.03cm} \mathfrak{q}\nsubseteq\mathfrak{p}\}$.\end{center}Given a homogeneous ideal $\mathfrak{a}$ in $R$, we set
\begin{center}$\mathcal{V}(\mathfrak{a})=\{\mathfrak{p}\in\textrm{Spec}R\hspace{0.03cm}|\hspace{0.03cm}\mathfrak{a}\subseteq\mathfrak{p}\}$.\end{center}Let $\mathcal{U}$ be a subset of $\mathrm{Spec}R$. The specialization closure of $\mathcal{U}$ is the set
\begin{center}$\mathrm{cl}\mathcal{U}=\{\mathfrak{p}\in\textrm{Spec}R\hspace{0.03cm}|\hspace{0.03cm}\textrm{there\ is}\ \mathfrak{q}\in\mathcal{U}\ \textrm{with}\ \mathfrak{q}\subseteq\mathfrak{p}\}$.\end{center}The subset $\mathcal{U}$ is
 specialization closed if $\mathrm{cl}\mathcal{U}=\mathcal{U}$. Note that the subsets $\mathcal{V}(\mathfrak{a})$ and $\mathcal{Z}(\mathfrak{p})$ are specialization closed.

From now on, $R$ denotes a commutative noetherian graded ring and $\mathscr{T}$ a compactly generated
$R$-linear triangulated category with set-indexed coproducts.

\vspace{2mm}
{\bf Local cohomology and homology.} An exact functor $L:\mathscr{T}\rightarrow \mathscr{T}$ is called localization functor if there exists a morphism
$\eta:\mathrm{Id}_\mathscr{T}\rightarrow L$ such that the morphism $L\eta:L\rightarrow L^2$ is invertible and $L\eta=\eta L$.
Let $\mathcal{V}$ be a specialization closed subset of $\mathrm{Spec}R$ and $L_\mathcal{V}$ the associated localization functor. By \cite[Definition 3.2]{BIK}, one then gets an exact
functor $\Gamma_\mathcal{V}$ on $\mathscr{T}$ and for each object $X$ an exact triangle
\begin{center}$\Gamma_\mathcal{V}X\rightarrow X\rightarrow L_\mathcal{V}X\rightsquigarrow$.\end{center}
We call $\Gamma_\mathcal{V}X$ the local cohomology of $X$ supported on $\mathcal{V}$, and the essential image of $\Gamma_\mathcal{V}$ is denoted by $\mathscr{T}_\mathcal{V}$.
By \cite[Corollary 6.5]{BIK}, the functors $L_\mathcal{V}$ and $\Gamma_\mathcal{V}$ on $\mathscr{T}$ preserve coproducts, and hence have right adjoints by Brown representability. Let $\Lambda^\mathcal{V}$ and $V^\mathcal{V}$ denote right adjoints of $\Gamma_\mathcal{V}$ and $V^\mathcal{V}$, respectively.
They induce a functorial exact triangle\begin{center}$V^\mathcal{V}X\rightarrow X\rightarrow \Lambda^\mathcal{V}X\rightsquigarrow$\end{center}
which gives rise to an exact local homology functor $\Lambda^\mathcal{V}:\mathscr{T}\rightarrow\mathscr{T}$. The essential image of $\Lambda^\mathcal{V}$ is denoted by $\mathscr{T}^\mathcal{V}$.

\vspace{2mm}
{\bf Koszul objects.}
Let $r\in R$ be a homogeneous element and $X$ an object in $\mathscr{T}$. We denote by $X/\hspace{-0.15cm}/r$
any object that appears in an exact triangle\begin{center}$X\xrightarrow{r}\Sigma^{|r|}X\rightarrow X/\hspace{-0.15cm}/r\rightsquigarrow$,\end{center}
call it a Koszul object of $r$ on $X$. It is well defined up to (non-unique) isomorphism. Given
a homogeneous ideal $\mathfrak{a}$ in $R$, we write $X/\hspace{-0.15cm}/\mathfrak{a}$ for any Koszul object obtained by iterating the
construction above with respect to some finite sequence of generators for $\mathfrak{a}$. This object may
depend on the choice of the minimal generating sequence for $\mathfrak{a}$.

\vspace{2mm}
{\bf Cohomologically finite objects.} For any objects $C\in\mathscr{T}^c$ and $X\in\mathscr{T}$, set $\mathrm{H}^\ast_{C}(X)=\mathrm{Hom}^\ast_\mathscr{T}(C,X)$. We say that an object $X$ in $\mathscr{T}$ is cohomologically finite with respect
to $C$ if $\mathrm{H}^\ast_{C}(X)$ is finitely generated as a graded $R$-module. $X$ is called cohomologically finite if it is cohomologically finite with respect to any $C\in\mathscr{T}^c$.
We also denote
\begin{center}$\mathrm{inf}_C(X)=\mathrm{inf}(\mathrm{H}^\ast_C(X))=\mathrm{inf}_C\{n\in\mathbb{Z}\hspace{0.03cm}|\hspace{0.03cm}\mathrm{H}^n_C(X)\neq0\}$,\end{center}
\begin{center}$\mathrm{sup}_C(X)=\mathrm{sup}(\mathrm{H}^\ast_C(X))=\mathrm{sup}_C\{n\in\mathbb{Z}\hspace{0.03cm}|\hspace{0.03cm}\mathrm{H}^n_C(X)\neq0\}$.\end{center}

\bigskip
\section{\bf Boundedness of local cohomology and homology}
In this section, we provide an explicit formulas for computing
boundedness of local cohomology and local homology using Koszul objects.

Let  $X_1\xrightarrow{u_1}X_2\xrightarrow{u_2}X_3\xrightarrow{u_3}\cdots$ be a sequence of morphisms in $\mathscr{T}$. Its homotopy colimit, denoted
by $\mathrm{hocolim}X_i$, is defined by an exact triangle
\begin{center}$\bigoplus_{i\geq 1}X_i\xrightarrow{\theta}\bigoplus_{i\geq 1}X_i\rightarrow\mathrm{hocolim}X_i\rightsquigarrow$,\end{center}
where $\theta$ is the map $(\mathrm{id}-u_i)$.

Fix a homogeneous element $r\in R$. For each object $X$ in $\mathscr{T}$, consider the following commutative diagram
\begin{center} $\xymatrix@C=23pt@R=20pt{
  X \ar@{=}[r]\ar[d]^r&X\ar@{=}[r]\ar[d]^{r^2}&X\ar@{=}[r]\ar[d]^{r^3} & \cdots \\
 \Sigma^{|r|}X \ar[d]\ar[r]^r& \Sigma^{|r^2|}X\ar[d]\ar[r]^r &\Sigma^{|r^3|}X\ar[d]\ar[r]^r&\cdots \\
 X/\hspace{-0.15cm}/r \ar[r]& X/\hspace{-0.15cm}/r^2\ar[r] & X/\hspace{-0.15cm}/r^3\ar[r]&\cdots}$
\end{center}
where each vertical sequence is given by the exact triangle defining $X/\hspace{-0.15cm}/r^n$, and the morphisms
in the last row are the ones induced by the commutativity of the upper squares.

\begin{lem}\label{lem:2.1}{\rm(\cite{BIK1})} {\it{Let $r\in R$ be a homogeneous element and $X$ an object in $\mathscr{T}$. Then the adjunction morphism $\Gamma_{\mathcal{V}(r)}X\rightarrow X$ induces an isomorphism
\begin{center}$\mathrm{hocolim}\Sigma^{-1}(X/\hspace{-0.15cm}/r^n)\stackrel{\sim}\longrightarrow\Gamma_{\mathcal{V}(r)}X$.\end{center}}}
\end{lem}

The following result provides a formula for computing $\mathrm{inf}_C(\Gamma_{\mathcal{V}(r)}X)$, which recovers part of \cite[Theorem 2.1]{FI}.

\begin{prop}\label{lem:2.2}{\it{Let $r\in R$ be a homogeneous element and $X$ an object in $\mathscr{T}$.

$\mathrm{(1)}$ If $|r|\leq0$, then $\mathrm{inf}_C(X/\hspace{-0.15cm}/r)+1=
\mathrm{inf}_C(\Gamma_{\mathcal{V}(r)}X)$ for any $C\in\mathscr{T}^c$.

$\mathrm{(2)}$ If $|r|>0$, then $\mathrm{inf}_C(X/\hspace{-0.15cm}/r)+|r|=\mathrm{inf}_C(\Gamma_{\mathcal{V}(r)}X)$ for any $C\in\mathscr{T}^c$.\\
In particular,
if $\mathfrak{a}$ is
a homogeneous ideal of $R$ then
\begin{center}$\mathrm{H}^{\ast}_C(\Gamma_{\mathcal{V}(\mathfrak{a})}X)\neq0\Longleftrightarrow\mathrm{H}^{\ast}_{C/\hspace{-0.1cm}/\mathfrak{a}}(X)\neq0
\Longleftrightarrow\mathrm{H}^{\ast}_C(X/\hspace{-0.15cm}/\mathfrak{a})\neq0,\ \forall\ C\in\mathscr{T}^c$.\end{center}}}
\end{prop}
\begin{proof} Assume that $|r|\leq0$.
 The octahedral axiom yields a commutative diagram of exact triangles in $\mathscr{T}$:\begin{center}$\xymatrix@C=20pt@R=20pt{
   X\ar@{=}[d]\ar[r]^{r\ \ } & \Sigma^{|r|}X\ar[d]^r\ar[r]& X/\hspace{-0.15cm}/r\ar[d]\ar[r] &\Sigma X \ar@{=}[d]\\
   X\ar[r]^{r^2\ \ } & \Sigma^{|r^2|}X\ar[d]\ar[r]& X/\hspace{-0.15cm}/r^2\ar[d]\ar[r] &\Sigma X\\
   & \Sigma^{|r|}(X/\hspace{-0.15cm}/r)\ar[d]\ar@{=}[r]& \Sigma^{|r|}(X/\hspace{-0.15cm}/r)\ar[d]\\
   & \Sigma^{|r|+1}X \ar[r]&\Sigma (X/\hspace{-0.15cm}/r)}$
   \end{center}
The third column triangle induces an exact sequence \begin{center}$\mathrm{H}^{i+|r|-1}_C(X/\hspace{-0.15cm}/r)\rightarrow\mathrm{H}^i_C(X/\hspace{-0.15cm}/r)\rightarrow\mathrm{H}^i_C(X/\hspace{-0.15cm}/r^2)
\rightarrow\mathrm{H}^{i+|r|}_C(X/\hspace{-0.15cm}/r)$,\end{center}which implies that $\mathrm{inf}_C(X/\hspace{-0.15cm}/r^2)\geq\mathrm{inf}_C(X/\hspace{-0.15cm}/r)$. By repeating this process, one has that $\mathrm{inf}_C(X/\hspace{-0.15cm}/r^n)\geq\mathrm{inf}_C(X/\hspace{-0.15cm}/r)$ for any $n\geq1$. Hence Lemma \ref{lem:2.1} implies that $\mathrm{inf}_C(\Gamma_{\mathcal{V}(r)}X)\geq\mathrm{inf}_C(X/\hspace{-0.15cm}/r)+1$. On the other hand, \cite[Lemma 2.6]{BIK1} gives us an exact triangle $\Gamma_{\mathcal{V}(r)}X\xrightarrow{r}\Sigma^{|r|}\Gamma_{\mathcal{V}(r)}X\rightarrow X/\hspace{-0.15cm}/r\rightsquigarrow$, which induces the following  exact sequence\begin{center}$\mathrm{H}^{i+|r|}_C(\Gamma_{\mathcal{V}(r)}X)\rightarrow\mathrm{H}^i_C(X/\hspace{-0.15cm}/r)
\rightarrow\mathrm{H}^{i+1}_C(\Gamma_{\mathcal{V}(r)}X)\rightarrow\mathrm{H}^{i+|r|+1}_C(\Gamma_{\mathcal{V}(r)}X)$.\end{center}
Therefore, one gets that $\mathrm{inf}_C(\Gamma_{\mathcal{V}(r)}X)\leq\mathrm{inf}_C(X/\hspace{-0.15cm}/r)+1$.

Assume $|r|>0$. Then the exact triangle $\Gamma_{\mathcal{V}(r)}X\xrightarrow{r}\Sigma^{|r|}\Gamma_{\mathcal{V}(r)}X\rightarrow X/\hspace{-0.15cm}/r\rightsquigarrow$ implies that $\mathrm{inf}_C(X/\hspace{-0.15cm}/r)+|r|=\mathrm{inf}_C(\Gamma_{\mathcal{V}(r)}X)$.

This completes the proof of the desired equality.
\end{proof}

\begin{cor}\label{lem:1.14}{\it{Let $r_1,r_2$ be homogeneous elements in $R$. For any object $X$ in $\mathscr{T}$, we have
\begin{center}$\mathrm{inf}_C(\Gamma_{\mathcal{V}(r_1)}X)\leq\mathrm{inf}_C(\Gamma_{\mathcal{V}(r_1,r_2)}X),\ \forall\ C\in\mathscr{T}^c$.\end{center}In particular, for any homogeneous ideals $\mathfrak{b}\subseteq\mathfrak{a}$ of $R$ and any $C\in\mathscr{T}^c$, one has that $\mathrm{inf}_C(\Gamma_{\mathcal{V}(\mathfrak{b})}X)\leq\mathrm{inf}_C(\Gamma_{\mathcal{V}(\mathfrak{a})}X)$.}}
\end{cor}
\begin{proof} The exact triangle $X/\hspace{-0.15cm}/r_1\xrightarrow{r_2}\Sigma^{|r_2|}X/\hspace{-0.15cm}/r_1\rightarrow X/\hspace{-0.15cm}/(r_1,r_2)\rightsquigarrow$ induces an exact sequence \begin{align}
\mathrm{H}^{i+|r_2|}_C(X/\hspace{-0.15cm}/r_1)\rightarrow\mathrm{H}^i_C(X/\hspace{-0.15cm}/(r_1,r_2))\rightarrow\mathrm{H}^{i+1}_C(X/\hspace{-0.15cm}/r_1)
\rightarrow\mathrm{H}^{i+|r_2|+1}_C(X/\hspace{-0.15cm}/r_1).
\label{exact04}\tag{\dag}\end{align}First assume that $|r_1|\leq0$.
If $|r_2|\leq0$, then the sequence $(\dag)$ implies that $\mathrm{inf}_C(X/\hspace{-0.15cm}/r_1)\leq\mathrm{inf}_C(X/\hspace{-0.15cm}/(r_1,r_2))+1$. Hence $\mathrm{inf}_C(\Gamma_{\mathcal{V}(r_1)}X)=\mathrm{inf}_C(X/\hspace{-0.15cm}/r_1)+1\leq\mathrm{inf}_C(X/\hspace{-0.15cm}/(r_1,r_2))+2=\mathrm{inf}_C(\Gamma_{\mathcal{V}(r_1,r_2)}X)$ by Proposition \ref{lem:2.2}(1). If $|r_2|>0$, the sequence $(\dag)$ implies that $\mathrm{inf}_C(X/\hspace{-0.15cm}/r_1)=\mathrm{inf}_C(X/\hspace{-0.15cm}/(r_1,r_2))+|r_2|$. So $\mathrm{inf}_C(\Gamma_{\mathcal{V}(r_1)}X)=\mathrm{inf}_C(X/\hspace{-0.15cm}/r_1)+1=\mathrm{inf}_C(X/\hspace{-0.15cm}/(r_1,r_2))+|r_2|+1
=\mathrm{inf}_C\Gamma_{\mathcal{V}(r_2)}(X/\hspace{-0.15cm}/r_1)+1=\mathrm{inf}_C(\Gamma_{\mathcal{V}(r_2)}X/\hspace{-0.15cm}/r_1)+1=\mathrm{inf}_C(\Gamma_{\mathcal{V}(r_1,r_2)}X)$ by Proposition \ref{lem:2.2}. Now assume that $|r_1|>0$.
If $|r_2|\leq0$, then $\mathrm{inf}_C(\Gamma_{\mathcal{V}(r_1)}X)=\mathrm{inf}_C(X/\hspace{-0.15cm}/r_1)+|r_1|\leq\mathrm{inf}_C(X/\hspace{-0.15cm}/(r_1,r_2))+|r_1|+1
=\mathrm{inf}_C(\Gamma_{\mathcal{V}(r_1)}(X/\hspace{-0.15cm}/r_2))+1=\mathrm{inf}_C(\Gamma_{\mathcal{V}(r_1,r_2)}X)$ by Proposition \ref{lem:2.2}.
If $|r_2|>0$, then $\mathrm{inf}_C(\Gamma_{\mathcal{V}(r_1)}X)=\mathrm{inf}_C(X/\hspace{-0.15cm}/r_1)+|r_1|=\mathrm{inf}_C(X/\hspace{-0.15cm}/(r_1,r_2))+|r_1|+|r_2|=\mathrm{inf}_C(\Gamma_{\mathcal{V}(r_1,r_2)}X)$ by Proposition \ref{lem:2.2}(2), as claimed.
\end{proof}

Since $\mathscr{T}$ is compactly generated with set-indexed coproducts, it follows from \cite[Proposition 8.4.6]{N} that $\mathscr{T}$ also admits set-indexed products. Let  $\cdots\xrightarrow{u_4}X_3\xrightarrow{u_3}X_2\xrightarrow{u_2}X_1$ be a sequence of morphisms in $\mathscr{T}$. Its homotopy limit, denoted
by $\mathrm{holim}X_i$, is defined by an exact triangle
\begin{center}$\mathrm{holim}X_i\rightarrow\prod_{i\geq 1}X_i\xrightarrow{\theta}\prod_{i\geq 1}X_i\rightsquigarrow$,\end{center}
where $\theta$ is the map $(\mathrm{id}-u_{i+1})$.

Fix a homogeneous element $r\in R$. For each object $X$ in $\mathscr{T}$, consider the following commutative diagram
\begin{center} $\xymatrix@C=23pt@R=20pt{
 \cdots\ar[r]^{r\ \ }& \Sigma^{-|r^3|}X\ar[d]^{r^3}\ar[r]^r &\Sigma^{-|r^2|}X\ar[d]^{r^2}\ar[r]^r&\Sigma^{-|r|}X\ar[d]^{r} \\
  \cdots \ar@{=}[r]&X\ar@{=}[r]\ar[d]&X\ar@{=}[r]\ar[d] & X \ar[d]\\
 \cdots \ar[r]& \Sigma^{-|r^3|}(X/\hspace{-0.15cm}/r^3)\ar[r] & \Sigma^{-|r^2|}(X/\hspace{-0.15cm}/r^2)\ar[r]&\Sigma^{-|r|}(X/\hspace{-0.15cm}/r)}$
\end{center}

The next result gives an explicit compute of $\Lambda^{\mathcal{V}(r)}X$.

\begin{lem}\label{lem:2.3}{\it{Let $r\in R$ be a homogeneous element and $X$ an object in $\mathscr{T}$. Then the adjunction morphism $X\rightarrow\Lambda^{\mathcal{V}(r)}X$ induces an isomorphism \begin{center}$\Lambda^{\mathcal{V}(r)}X\stackrel{\sim}\longrightarrow\mathrm{holim}\Sigma^{-|r^n|}(X/\hspace{-0.15cm}/r^n)$.\end{center}}}
\end{lem}
\begin{proof} By \cite[Lemma 4.12]{BIK2} each $\Sigma^{-|r^n|}(X/\hspace{-0.15cm}/r^n)$ is in $\mathscr{T}^{\mathcal{V}(r)}$, and hence so is $\mathrm{holim}\Sigma^{-|r^n|}(X/\hspace{-0.15cm}/r^n)$. So $\mathrm{holim}\Sigma^{-|r^n|}(X/\hspace{-0.15cm}/r^n)\stackrel{\sim}\rightarrow\Lambda^{\mathcal{V}(r)}\mathrm{holim}\Sigma^{-|r^n|}(X/\hspace{-0.15cm}/r^n)$. Fix $C\in\mathscr{T}^c$. By \cite[Proposition 2.9]{BIK1},  \begin{center}$\mathrm{Hom}^\ast_\mathscr{T}(C,\Lambda^{\mathcal{V}(r)}\mathrm{holim}\Sigma^{-|r^n|}X)
\cong\mathrm{Hom}^\ast_\mathscr{T}(\Gamma_{\mathcal{V}(r)}\mathrm{hocolim}\Sigma^{|r^n|}C,X)$=0,\end{center}which implies that $\Lambda^{\mathcal{V}(r)}X\stackrel{\sim}\rightarrow\Lambda^{\mathcal{V}(r)}\mathrm{holim}\Sigma^{-|r^n|}(X/\hspace{-0.15cm}/r^n)$.
This shows our claim.
\end{proof}

By analogy with the proof of Proposition \ref{lem:2.2}, one can give the foliowing formula for computing $\mathrm{sup}_C(\Lambda^{\mathcal{V}(r)}X)$, which recovers part of \cite[Theorem 4.1]{FI}.

\begin{prop}\label{lem:2.4}{\it{Let $r\in R$ be a homogeneous element and $X$ an object in $\mathscr{T}$.

$\mathrm{(1)}$ If $|r|\leq0$, then $\mathrm{sup}_C(X/\hspace{-0.15cm}/r)+|r|=
\mathrm{sup}_C(\Lambda^{\mathcal{V}(r)}X)$ for any $C\in\mathscr{T}^c$.

$\mathrm{(2)}$ If $|r|>0$, then $\mathrm{sup}_C(X/\hspace{-0.15cm}/r)+1=
\mathrm{sup}_C(\Lambda^{\mathcal{V}(r)}X)$ for any $C\in\mathscr{T}^c$.\\
In particular,
if $\mathfrak{a}$ is
a homogeneous ideal of $R$ then
\begin{center}$\mathrm{H}^{\ast}_C(\Lambda^{\mathcal{V}(\mathfrak{a})}X)\neq0\Longleftrightarrow\mathrm{H}^{\ast}_{C/\hspace{-0.1cm}/\mathfrak{a}}(X)\neq0
\Longleftrightarrow\mathrm{H}^{\ast}_C(X/\hspace{-0.15cm}/\mathfrak{a})\neq0$.\end{center}}}
\end{prop}

\begin{cor}\label{lem:1.100}{\it{Let $\mathfrak{a}$ be a homogeneous ideal of $R$ and $X$ an object in $\mathscr{T}$. Then

$\mathrm{(1)}$ $\mathrm{inf}_C(\Gamma_{\mathcal{V}(\mathfrak{a})}X)-\mathrm{inf}_C(X)\geq0$ for any $C\in\mathscr{T}^c$.

$\mathrm{(2)}$ $\mathrm{sup}_C(\Lambda^{\mathcal{V}(\mathfrak{a})}X)-\mathrm{sup}_C(X)\leq0$ for any $C\in\mathscr{T}^c$.

$\mathrm{(3)}$ $\mathrm{sup}_C(\Gamma_{\mathcal{V}(\mathfrak{a})}X)\geq\mathrm{sup}_C(\Lambda^{\mathcal{V}(\mathfrak{a})}X)$ for any $C\in\mathscr{T}^c$.

$\mathrm{(4)}$ $\mathrm{inf}_C(\Lambda^{\mathcal{V}(\mathfrak{a})}X)\leq\mathrm{inf}_C(\Gamma_{\mathcal{V}(\mathfrak{a})}X)$ for any $C\in\mathscr{T}^c$.\\
In particular, $\Gamma_{\mathcal{V}(\mathfrak{a})}X\neq0$ if and only if $\Lambda^{\mathcal{V}(\mathfrak{a})}X\neq0$.}}
\end{cor}
\begin{proof} We only prove (1) and (3) since the proof of (2) and (4) are similar.

(1) If $|r|\leq0$, then $\mathrm{inf}_C(X/\hspace{-0.15cm}/r)\geq\mathrm{inf}_C(X)-1$. Hence Proposition \ref{lem:2.2}(1) implies that $\mathrm{inf}_C(\Gamma_{\mathcal{V}(r)}X)-\mathrm{inf}_C(X)\geq0$. If $|r|>0$, then $\mathrm{inf}_C(X/\hspace{-0.15cm}/r)=\mathrm{inf}_C(X)-|r|$. Hence Proposition \ref{lem:2.2}(2) implies that $\mathrm{inf}_C(\Gamma_{\mathcal{V}(r)}X)-\mathrm{inf}_C(X)=0$, as desired.

(3) Consider the exact triangle $\Gamma_{\mathcal{V}(r)}X\xrightarrow{r}\Sigma^{|r|}\Gamma_{\mathcal{V}(r)}X\rightarrow X/\hspace{-0.15cm}/r\rightsquigarrow$.
If $|r|\leq0$, then $\mathrm{sup}_C(X/\hspace{-0.15cm}/r)\leq\mathrm{sup}_C(\Gamma_{\mathcal{V}(r)}X)-|r|$. Hence Proposition \ref{lem:2.4}(1) implies that $\mathrm{sup}_C(\Gamma_{\mathcal{V}(r)}X)\geq\mathrm{sup}_C(\Lambda^{\mathcal{V}(r)}X)$. If $|r|>0$, then $\mathrm{sup}_C(X/\hspace{-0.15cm}/r)=\mathrm{sup}_C(\Gamma_{\mathcal{V}(r)}X)-1$. Hence Proposition \ref{lem:2.4}(2) implies that
$\mathrm{sup}_C(\Gamma_{\mathcal{V}(r)}X)=\mathrm{sup}_C(\Lambda^{\mathcal{V}(r)}X)$, as desired.
\end{proof}

\begin{rem}\label{lem:1.8}{\rm (1) Let $R$ be a commutative noetherian $\mathbb{N}_0$-graded ring and $\mathscr{T}=\mathrm{D}(R)$ the derived category of DG-modules over $R$. Then $\mathfrak{m}=(\bigoplus_{i\geq1}R_i)\oplus\mathfrak{m}_0$ is the maximal ideal of $R$ where $(R_0,\mathfrak{m}_0)$ is a local ring. The number
$\mathrm{inf}(\Gamma_{\mathcal{V}(\mathfrak{m})}X)$ is exactly the depth of a DG-module $X$ introduced by Shaul in \cite{S}.

(2) Let $A$ be a commutative noetherian ring and $\mathscr{T}=\mathrm{D}(A)$ the derived category of $A$-complexes. Let $\mathfrak{a}$ be
an ideal in $A$ and
 $K$ the Koszul complex on a sequence of $n$ generators for $\mathfrak{a}$. Then
Proposition \ref{lem:2.2} implies that $\mathrm{inf}(\Gamma_{\mathcal{V}(\mathfrak{a})}X)=\mathrm{inf}(K\otimes_AX)+n$, this common value is exactly the $\mathfrak{a}$-depth of a complex $X$ introduced by Foxby and Iyengar in \cite{FI}. Also Proposition \ref{lem:2.4} implies that $-\mathrm{sup}(\Lambda^{\mathcal{V}(\mathfrak{a})}Y)=-\mathrm{sup}(K\otimes_AY)$, this common value is exactly the $\mathfrak{a}$-width of a complex $Y$ introduced by Foxby and Iyengar in \cite{FI}.}
\end{rem}

\bigskip
\section{\bf Cosupport}
Since cosupport is not as well understood as support, we devote this section to some characterizations of cosupport, and give the proof of Theorem A.

For each $\mathfrak{p}$ in $\mathrm{Spec}R$, denote the exact functor $V^{\mathcal{Z}(\mathfrak{p})}\Lambda^{\mathcal{V}(\mathfrak{p})}:\mathscr{T}\rightarrow\mathscr{T}$ by $\Lambda^\mathfrak{p}$. The essential image of $\Lambda^\mathfrak{p}$ is denoted by $\mathscr{T}^{\{\mathfrak{p}\}}$, it is a colocalizing subcategory of $\mathscr{T}$.

The cosupport of an object $X$ in $\mathscr{T}$ is a subset of $\mathrm{Spec}R$ defined as follows:
\begin{center}$\mathrm{cosupp}_RX=\{\mathfrak{p}\in \mathrm{Spec}R\hspace{0.03cm}|\hspace{0.03cm}\Lambda^\mathfrak{p}X\neq0\}$.\end{center}

For
each compact object $C$ in $\mathscr{T}$ and each injective $R$-module $I$, Brown representability
yields an object $T_C(I)$ in $\mathscr{T}$ and a natural isomorphism:
\begin{center}$\mathrm{Hom}^\ast_\mathscr{T}(-,T_C(I))=\mathrm{Hom}^\ast_R(\mathrm{H}^\ast_C(-),I)$.\end{center}
For $\mathfrak{p}\in\mathrm{Spec}R$, let $I(\mathfrak{p})$ denote the injective envelope of $R/\mathfrak{p}$. Then $\Sigma^iI(\mathfrak{p})$ cogenerate the category of $\mathfrak{p}$-local $R$-modules. For any object $X$ in $\mathscr{T}$, denote
$X(\mathfrak{p})=X_\mathfrak{p}/\hspace{-0.15cm}/\mathfrak{p}$.

Next we give an axiomatic description of cosupport via Koszul objects, analogous to the one for support in \cite[Proposition 5.12]{BIK}.

\begin{prop}\label{lem:2.13}{\it{Let $X$ be an object in $\mathscr{T}$. For each $\mathfrak{p}\in\mathrm{Spec}R$ and $C\in\mathscr{T}^c$, the following conditions are equivalent:

$\mathrm{(1)}$ $\mathrm{Hom}^\ast_\mathscr{T}(C,\Lambda^\mathfrak{p}X)\neq0$;

$\mathrm{(2)}$ $\mathrm{Hom}^\ast_\mathscr{T}(C/\hspace{-0.15cm}/\mathfrak{p},\Lambda^\mathfrak{p}X)\neq0$;

$\mathrm{(3)}$ $\mathrm{Hom}^\ast_\mathscr{T}(C/\hspace{-0.15cm}/\mathfrak{p},V^{\mathcal{Z}(\mathfrak{p})}X)\neq0$;

$\mathrm{(4)}$ $\mathrm{Hom}^\ast_\mathscr{T}(C/\hspace{-0.15cm}/\mathfrak{p},\Gamma_{\mathcal{V}(\mathfrak{p})}V^{\mathcal{Z}(\mathfrak{p})}X)\neq0$;

$\mathrm{(5)}$ $\mathrm{Hom}^\ast_\mathscr{T}(C(\mathfrak{p}),X)\neq0$;

$\mathrm{(6)}$ $\mathrm{Hom}^\ast_\mathscr{T}(\Lambda^\mathfrak{p}X,T_{C}(I(\mathfrak{p})))\neq0$;

$\mathrm{(7)}$ $\mathrm{Hom}^\ast_\mathscr{T}(\Lambda^\mathfrak{p}X,T_{C/\hspace{-0.1cm}/\mathfrak{p}}(I(\mathfrak{p})))\neq0$;

$\mathrm{(8)}$ $\mathrm{Hom}^\ast_\mathscr{T}(V^{\mathcal{Z}(\mathfrak{p})}X,T_{C/\hspace{-0.1cm}/\mathfrak{p}}(I(\mathfrak{p})))\neq0$.}}
\end{prop}
\begin{proof} (1) $\Leftrightarrow$ (2) $\mathrm{H}^\ast_C(\Lambda^\mathfrak{p}X)\neq0$ if and only if $\mathrm{H}^\ast_C(\Gamma_{\mathcal{V}(\mathfrak{p})}V^{\mathcal{Z}(\mathfrak{p})}X)\neq0$ by Propositions \ref{lem:2.2} and \ref{lem:2.4} if and only if $\mathrm{H}^\ast_{C/\hspace{-0.1cm}/\mathfrak{p}}(\Gamma_{\mathcal{V}(\mathfrak{p})}V^{\mathcal{Z}(\mathfrak{p})}X)\neq0$ since $\mathrm{H}^\ast_C(\Gamma_{\mathcal{V}(\mathfrak{p})}V^{\mathcal{Z}(\mathfrak{p})}X)$ is  $\mathfrak{p}$-torsion if and only if $\mathrm{H}^\ast_{C/\hspace{-0.1cm}/\mathfrak{p}}(\Lambda^\mathfrak{p}X)\neq0$ by Propositions \ref{lem:2.2} and \ref{lem:2.4} again.

(2) $\Leftrightarrow$ (3) By $(\Gamma_{\mathcal{V}(\mathfrak{p})},\Lambda^{\mathcal{V}(\mathfrak{p})})$ is an adjoint pair and $\Gamma_{\mathcal{V}(\mathfrak{p})}(C/\hspace{-0.15cm}/\mathfrak{p})\cong C/\hspace{-0.15cm}/\mathfrak{p}$.

(3) $\Leftrightarrow$ (4) This follows from that $\Gamma_{\mathcal{V}(\mathfrak{p})}$ is a right adjoint of the inclusion $\mathscr{T}_{\mathcal{V}(\mathfrak{p})}\rightarrow\mathscr{T}$.

(3) $\Leftrightarrow$ (5) By $(L_{\mathcal{Z}(\mathfrak{p})},V^{\mathcal{Z}(\mathfrak{p})})$ is an adjoint pair.

(1) $\Leftrightarrow$ (6) and (2) $\Leftrightarrow$ (7) follow from the fact that $\Sigma^iI(\mathfrak{p})$ cogenerate the category of $\mathfrak{p}$-local $R$-modules and the graded $R$-modules $\mathrm{H}^\ast_{C}(\Lambda^\mathfrak{p}X)$ and $\mathrm{H}^\ast_{C/\hspace{-0.1cm}/\mathfrak{p}}(\Lambda^\mathfrak{p}X)$ are $\mathfrak{p}$-local.

(7) $\Leftrightarrow$ (8) It follows from that $\Lambda^{\mathcal{V}(\mathfrak{p})}$ is a left adjoint of the inclusion $\mathscr{T}^{\mathcal{V}(\mathfrak{p})}\rightarrow\mathscr{T}$ and the object $T_{C/\hspace{-0.1cm}/\mathfrak{p}}(I(\mathfrak{p}))$ is in $\mathscr{T}^{\mathcal{V}(\mathfrak{p})}$.
\end{proof}

The next result recovers part of \cite[Proposition 4.4]{WW}.

\begin{prop}\label{lem:2.6}{\it{Let $X$ be an object in $\mathscr{T}$ and $\mathfrak{p}\in\mathrm{Spec} R$. The following are equivalent:

$\mathrm{(1)}$ $\mathfrak{p}\in\mathrm{cosupp}_R(X)$;

$\mathrm{(2)}$ $V^{\mathcal{Z}(\mathfrak{p})}(X/\hspace{-0.15cm}/\mathfrak{p})\neq0$;

$\mathrm{(3)}$ $\Gamma_{\mathcal{V}(\mathfrak{p})}V^{\mathcal{Z}(\mathfrak{p})}X\neq0$;

$\mathrm{(4)}$ $\mathfrak{p}\in\mathrm{cosupp}_R(V^{\mathcal{Z}(\mathfrak{p})}X)$.}}
\end{prop}
\begin{proof} (1) $\Leftrightarrow$ (2) $\Lambda^{\mathcal{V}(\mathfrak{p})}V^{\mathcal{Z}(\mathfrak{p})}X\neq0$ if and only if $\mathrm{H}^\ast_C(\Lambda^{\mathcal{V}(\mathfrak{p})}V^{\mathcal{Z}(\mathfrak{p})}X)\neq0$ for some $C\in\mathscr{T}^c$ if and only if $\mathrm{H}^\ast_{C}((V^{\mathcal{Z}(\mathfrak{p})}X)/\hspace{-0.15cm}/\mathfrak{p})\neq0$ for some $C\in\mathscr{T}^c$ by Proposition \ref{lem:2.4} if and only if $V^{\mathcal{Z}(\mathfrak{p})}(X/\hspace{-0.15cm}/\mathfrak{p})\neq0$ since $(V^{\mathcal{Z}(\mathfrak{p})}X)/\hspace{-0.15cm}/\mathfrak{p}\cong V^{\mathcal{Z}(\mathfrak{p})}(X/\hspace{-0.15cm}/\mathfrak{p})$.

(1) $\Leftrightarrow$ (3) This follows from Corollary \ref{lem:1.100}.

(1) $\Leftrightarrow$ (4) By \cite[Corollary 4.9]{BIK2}, $\mathrm{cosupp}_R(V^{\mathcal{Z}(\mathfrak{p})}X)\subseteq\mathcal{U}(\mathfrak{p})$. Fix a prime $\mathfrak{q}\in\mathrm{Spec}R$ with $\mathfrak{q}\subseteq\mathfrak{p}$. Then $\Lambda^\mathfrak{q}(V^{\mathcal{Z}(\mathfrak{p})}X)\cong\Lambda^\mathfrak{q}X$ by \cite[(4.2)]{BIK2}, and hence $\mathrm{cosupp}_R(V^{\mathcal{Z}(\mathfrak{p})}X)=\mathcal{U}(\mathfrak{p})\cap\mathrm{cosupp}_RX$. This means that (1) $\Leftrightarrow$ (4) holds.
\end{proof}

Analogous to \cite[Theorem 5.2]{BIK}, we give the following computation of cosupport.

\begin{thm}\label{lem:4.40}{\it{For each object $X$ in $\mathscr{T}$, one has an equality
 \begin{center}$\mathrm{cosupp}_{R}X=\bigcup_{C\in\mathscr{T}^c,\ \mathfrak{p}\in\mathrm{Spec}R}\mathrm{cosupp}_R\mathrm{Hom}^\ast_\mathscr{T}(V^{\mathcal{Z}(\mathfrak{p})}X,T_{C/\hspace{-0.1cm}/\mathfrak{p}}(I(\mathfrak{p})))$.\end{center}}}
\end{thm}
\begin{proof} By \cite[Theorem 4.5 and Proposition 4.4]{BIK2}, $X=0$ if and only if $\mathrm{H}^\ast_{C/\hspace{-0.1cm}/\mathfrak{p}}(V^{\mathcal{Z}(\mathfrak{p})}X)\cong\mathrm{Hom}^\ast_\mathscr{T}(C(\mathfrak{p}),X)=0$ for all $C\in\mathscr{T}^c$ and $\mathfrak{p}\in\mathrm{Spec}R$ if and only if $\mathrm{Hom}^\ast_\mathscr{T}(V^{\mathcal{Z}(\mathfrak{p})}X,T_{C/\hspace{-0.1cm}/\mathfrak{p}}(I(\mathfrak{p})))=0$ for all $C\in\mathscr{T}^c$ and $\mathfrak{p}\in\mathrm{Spec}R$. So assume that $X\neq0$.
 Let $\mathfrak{p}$ be a point in $\mathrm{Spec}R$ such that $\Lambda^\mathfrak{p}X\neq0$. Then there is $C\in\mathscr{T}^c$ such that $\mathrm{Hom}^\ast_\mathscr{T}(V^{\mathcal{Z}(\mathfrak{p})}X,T_{C/\hspace{-0.1cm}/\mathfrak{p}}(I(\mathfrak{p})))\neq0$ by Proposition \ref{lem:2.13}. Hence \cite[Proposition 4.10]{WW}  implies that $\mathrm{cosupp}_R\mathrm{Hom}^\ast_\mathscr{T}(V^{\mathcal{Z}(\mathfrak{p})}X,T_{C/\hspace{-0.1cm}/\mathfrak{p}}(I(\mathfrak{p})))=
\mathrm{supp}_R\mathrm{H}^\ast_{C/\hspace{-0.1cm}/\mathfrak{p}}(V^{\mathcal{Z}(\mathfrak{p})}X)\cap\mathrm{cosupp}_RT_{C/\hspace{-0.1cm}/\mathfrak{p}}(I(\mathfrak{p}))=\{\mathfrak{p}\}$. This justifies the inclusion $\mathrm{cosupp}_{R}X\subseteq\bigcup_{C\in\mathscr{T}^c,\ \mathfrak{p}\in\mathrm{Spec}R}\mathrm{cosupp}_R\mathrm{Hom}^\ast_\mathscr{T}(V^{\mathcal{Z}(\mathfrak{p})}X,T_{C/\hspace{-0.1cm}/\mathfrak{p}}(I(\mathfrak{p})))$.
Let now $\mathfrak{p}\in\mathrm{Spec}R$. Then there is an object $C$ in $\mathscr{T}^c$ so that $\mathrm{cosupp}_R\mathrm{Hom}^\ast_\mathscr{T}(V^{\mathcal{Z}(\mathfrak{p})}X,T_{C/\hspace{-0.1cm}/\mathfrak{p}}(I(\mathfrak{p})))=
\{\mathfrak{p}\}$, and so $\mathrm{Hom}^\ast_\mathscr{T}(C(\mathfrak{p}),X)\neq0$. Consequently, $\mathfrak{p}\in\mathrm{cosupp}_RX$ by \cite[Proposition 4.4]{BIK2}. This completes the proof.
\end{proof}

\begin{cor}\label{lem:1.10}{\it{For any specialization closed subset
$\mathcal{V}$ of $\mathrm{Spec}R$, one has that
\begin{center}$\begin{aligned}\mathscr{T}^\mathcal{V}
&=\{X\in\mathscr{T}\hspace{0.03cm}|\hspace{0.03cm}\mathrm{cosupp}_RX\subseteq\mathcal{V}\}\\
&=\{X\in\mathscr{T}\hspace{0.03cm}|\hspace{0.03cm}\mathrm{cosupp}_R\mathrm{Hom}^\ast_\mathscr{T}(V^{\mathcal{Z}(\mathfrak{p})}X,T_{C}(I(\mathfrak{p})))\subseteq\mathcal{V},\ \forall\ C\in\mathscr{T}^c,\ \forall\ \mathfrak{p}\in\mathcal{V}\}.\end{aligned}$\end{center}}}
\end{cor}
\begin{proof} By Theorem \ref{lem:4.40}, $\mathrm{cosupp}_RX\subseteq\mathcal{V}$ if and only if $\mathrm{supp}_R\mathrm{H}^\ast_{C/\hspace{-0.1cm}/\mathfrak{p}}(V^{\mathcal{Z}(\mathfrak{p})}X)\subseteq\mathcal{V}$ for all $C\in\mathscr{T}^c$ and $\mathfrak{p}\in\mathrm{Spec}R$. Hence \cite[Theorem 5.13]{BIK} implies that $\mathrm{supp}_R\mathrm{H}^\ast_{C}(V^{\mathcal{Z}(\mathfrak{p})}X)\subseteq\mathcal{V}$ for all $C\in\mathscr{T}^c$  and $\mathfrak{p}\in\mathrm{Spec}R$.
 But $\mathrm{cosupp}_R\mathrm{Hom}^\ast_\mathscr{T}(V^{\mathcal{Z}(\mathfrak{p})}X,T_{C}(I(\mathfrak{p})))=\mathrm{supp}_R\mathrm{H}^\ast_{C}(V^{\mathcal{Z}(\mathfrak{p})}X)$ for all $C\in\mathscr{T}^c$  and $\mathfrak{p}\in\mathrm{Spec}R$, so the proof is complete.
\end{proof}

\begin{cor}\label{lem:5.15}{\it{For each specialization closed subset $\mathcal{V}\subseteq\mathrm{Spec}R$ and each object $X$ in $\mathscr{T}$, one has that
\begin{center}$\mathrm{cosupp}_{R}(\Gamma_{\mathcal{V}}X)\cap\mathcal{V}=\mathrm{cosupp}_{R}X\cap\mathcal{V}$.\end{center}In particular, $\Gamma_{\mathcal{V}}X=0$ if and only if $\mathrm{cosupp}_{R}X\cap\mathcal{V}=\emptyset$.}}
\end{cor}
\begin{proof} By Theorem \ref{lem:4.40}, we have the following equalities
\begin{center}$\begin{aligned}\mathrm{cosupp}_{R}(\Gamma_{\mathcal{V}}X)\cap\mathcal{V}
&=\bigcup_{C\in\mathscr{G},\mathfrak{p}\in\mathcal{V}}\mathrm{cosupp}_R\mathrm{Hom}^\ast_\mathscr{T}(V^{\mathcal{Z}(\mathfrak{p})}\Gamma_{\mathcal{V}}X,T_{C/\hspace{-0.1cm}/\mathfrak{p}}(I(\mathfrak{p})))\\
&=\bigcup_{C\in\mathscr{G},\mathfrak{p}\in\mathcal{V}}\mathrm{cosupp}_R\mathrm{Hom}^\ast_R(\mathrm{Hom}^\ast_\mathscr{T}(C(\mathfrak{p}),\Gamma_{\mathcal{V}}X),I(\mathfrak{p}))\\
&=\bigcup_{C\in\mathscr{G},\mathfrak{p}\in\mathcal{V}}\mathrm{cosupp}_R\mathrm{Hom}^\ast_R(\mathrm{Hom}^\ast_\mathscr{T}(C/\hspace{-0.15cm}/\mathfrak{p},V^{\mathcal{Z}(\mathfrak{p})}X),I(\mathfrak{p}))\\
&=\mathrm{cosupp}_{R}X\cap\mathcal{V}.\end{aligned}$\end{center}We obtain the equality we seek.

If $\Gamma_{\mathcal{V}}X=0$, then $\mathrm{cosupp}_{R}X\cap\mathcal{V}=\emptyset$. Assume that $\mathrm{cosupp}_{R}X\cap\mathcal{V}=\emptyset$. Then $\mathrm{cosupp}_{R}X\subseteq\mathrm{Spec}R\backslash\mathcal{V}$, and so $X\cong L_{\mathcal{V}}X$ by \cite[Corollary 4.9]{BIK2}. Consequently, $\Gamma_{\mathcal{V}}X\cong\Gamma_{\mathcal{V}}L_{\mathcal{V}}X=0$.
\end{proof}

\begin{prop}\label{lem:2.12}{\it{For any $\mathfrak{p}\in\mathrm{Spec}R$ and any object $X$ in $\mathscr{T}$, one has an exact triangle
 \begin{center}$X'\rightarrow X\rightarrow X''\rightsquigarrow$\end{center} where $X'\in{^\bot}(\mathscr{T}^{\{\mathfrak{p}\}})$ and $X''\in\mathscr{T}^{\{\mathfrak{p}\}}$.}}
\end{prop}
\begin{proof} Let $X$ be an object in $\mathscr{T}$. The exact triangle $\Gamma_{\mathcal{Z}(\mathfrak{p})}X\rightarrow X\rightarrow L_{\mathcal{Z}(\mathfrak{p})}X\rightsquigarrow$ induces an exact triangle
$\Lambda^{\mathcal{V}(\mathfrak{p})}\Gamma_{\mathcal{Z}(\mathfrak{p})}X\rightarrow \Lambda^{\mathcal{V}(\mathfrak{p})}X\rightarrow\Lambda^{\mathcal{V}(\mathfrak{p})}L_{\mathcal{Z}(\mathfrak{p})}X\rightsquigarrow$. By the octahedral axiom, one has a commutative diagram of triangles in $\mathscr{T}$\begin{center}$\xymatrix@C=25pt@R=20pt{
  & \Lambda^{\mathcal{V}(\mathfrak{p})}\Gamma_{\mathcal{Z}(\mathfrak{p})}X \ar[d]\ar@{=}[r]& \Lambda^{\mathcal{V}(\mathfrak{p})}\Gamma_{\mathcal{Z}(\mathfrak{p})}X\ar[d]&\\
   X\ar@{=}[d]\ar[r] & \Lambda^{\mathcal{V}(\mathfrak{p})}X\ar[d]\ar[r]& \Sigma V^{\mathcal{V}(\mathfrak{p})}X\ar[d]\ar[r] &\Sigma X \ar@{=}[d]\\
   X\ar[r] & \Lambda^{\mathcal{V}(\mathfrak{p})}L_{\mathcal{Z}(\mathfrak{p})}X\ar[d]\ar[r]& \Sigma X' \ar[d]\ar[r] &\Sigma X\\
   & \Sigma \Lambda^{\mathcal{V}(\mathfrak{p})}\Gamma_{\mathcal{Z}(\mathfrak{p})}X \ar@{=}[r]&\Sigma \Lambda^{\mathcal{V}(\mathfrak{p})}\Gamma_{\mathcal{Z}(\mathfrak{p})}X&}$\end{center}
 Set $\Lambda^{\mathcal{V}(\mathfrak{p})}L_{\mathcal{Z}(\mathfrak{p})}X=X''$.  Since $\Lambda^\mathfrak{p}X''\cong X''$ by \cite[Proposition 2.3(3)]{BIK2}, it follows that $X''\in\mathscr{T}^{\{\mathfrak{p}\}}$ by \cite[page 176]{BIK2}.
On the other hand, one has the following isomorphisms \begin{center}$\begin{aligned}\mathrm{Hom}^\ast_\mathscr{T}(\Lambda^{\mathcal{V}(\mathfrak{p})}L_{\mathcal{Z}(\mathfrak{p})}X,T_{C/\hspace{-0.1cm}/\mathfrak{p}}(I(\mathfrak{p})))
&\cong\mathrm{Hom}^\ast_\mathscr{T}(L_{\mathcal{Z}(\mathfrak{p})}X,T_{C/\hspace{-0.1cm}/\mathfrak{p}}(I(\mathfrak{p})))\\
&\cong\mathrm{Hom}^\ast_\mathscr{T}(X,T_{C/\hspace{-0.1cm}/\mathfrak{p}}(I(\mathfrak{p}))),\ \forall\ C\in\mathscr{T}^c.\end{aligned}$\end{center}
Hence
$\mathrm{Hom}^\ast_\mathscr{T}(X',T_{C/\hspace{-0.1cm}/\mathfrak{p}}(I(\mathfrak{p})))=0$ for all $C\in\mathscr{T}^c$, and therefore $Y\in{^\bot}(\mathscr{T}^{\{\mathfrak{p}\}})$ by \cite[Proposition 5.4]{BIK2}, as claimed.
\end{proof}

\bigskip
\section{\bf Comparison of support and cosupport}
 The task of this section is to
compare the support and cosupport of cohomologically finite objects, and give the proof of Theorem B. For any cohomologically finite object $X$, we show that $\mathrm{supp}_RX=\mathrm{Supp}_RX$, obtain some applications of this equality.

Consider the exact functor $\Gamma_\mathfrak{p}:\mathscr{T}\rightarrow\mathscr{T}$ defined by
$\Gamma_\mathfrak{p}X=\Gamma_{\mathcal{V}(\mathfrak{p})}L_{\mathcal{Z}(\mathfrak{p})}X$.
The essential image of $\Gamma_\mathfrak{p}$ is denoted by $\mathscr{T}_{\{\mathfrak{p}\}}$, it is a localizing subcategory of $\mathscr{T}$.

The support of an object $X$ in $\mathscr{T}$ is a subset of $\mathrm{Spec}R$ defined as follows:
\begin{center}$\mathrm{supp}_RX=\{\mathfrak{p}\in \mathrm{Spec}R\hspace{0.03cm}|\hspace{0.03cm}\Gamma_\mathfrak{p}X\neq0\}$.\end{center}
For any graded $R$-module $M$, set $\mathrm{Supp}_RM:=\{\mathfrak{p}\in\mathrm{Spec}R\hspace{0.03cm}|\hspace{0.03cm}M_\mathfrak{p}\neq0\}$. This subset is sometimes referred to as the `big support' of $M$ to distinguish it from its `homological' support,
$\mathrm{supp}_RM$. For any object $X$ in $\mathscr{T}$, denote
$X_\mathfrak{p}=L_{\mathcal{Z}(\mathfrak{p})}X$. \cite[Theorem 3.3]{ASS} proved that
\begin{center}$\mathrm{Supp}_RX:=\{\mathfrak{p}\in\mathrm{Spec}R\hspace{0.03cm}|\hspace{0.03cm}X_\mathfrak{p}\neq0\}=\bigcup_{C\in\mathscr{T}^c}\mathrm{Supp}_R\mathrm{H}^\ast_C(X)$.\end{center}

\begin{lem}\label{lem:4.5}{\it{For each object $X$ of $\mathscr{T}$ there is an inclusion of sets $\mathrm{supp}_RX\subseteq\mathrm{Supp}_RX$; equality holds if $X$ is cohomologically finite.}}
\end{lem}
\begin{proof} The containment $\mathrm{supp}_RX\subseteq\mathrm{Supp}_RX$ is clear by definition. If $X$ is cohomologically finite, then $\mathrm{Supp}_RX=\bigcup_{C\in\mathscr{T}^c}\mathrm{Supp}_R\mathrm{H}^\ast_{C}(X)=\bigcup_{C\in\mathscr{T}^c}\mathrm{supp}_R\mathrm{H}^\ast_{C}(X)\subseteq\mathrm{supp}_RX$ by \cite[Theorem 5.5]{BIK}, as claimed.
\end{proof}

Let $\mathcal{U}$ be a subset of $\mathrm{Spec}R$. We denote by $\mathrm{min}\hspace{0.05cm}\mathcal{U}$ the set of minimal elements with respect to inclusion in $\mathcal{U}$. The dimension of a subset $\mathcal{U}$ of $\mathrm{Spec}R$, denoted by $\mathrm{dim}\hspace{0.05cm}\mathcal{U}$, is the supremum of all integers $n$
such that there exists a chain $\mathfrak{p}_0\subsetneq\mathfrak{p}_1\subsetneq\cdots\subsetneq\mathfrak{p}_n$ in $\mathcal{U}$. The set $\mathcal{U}$ is called discrete if $\mathrm{dim}\hspace{0.05cm}\mathcal{U}=0$.

Our goal of the following results study when the inclusion $\mathrm{cosupp}_{R}X\subseteq\mathrm{supp}_{R}X$ holds.

\begin{thm}\label{lem:5.13}{\it{Let $X$ be a cohomologically finite object in $\mathscr{T}$.
If $\mathrm{min}(\mathrm{supp}_{R}X)$ is finite, then there is an inequality
\begin{center}$\mathrm{cosupp}_{R}X\subseteq\mathrm{supp}_{R}X$.\end{center}}}
\end{thm}
\begin{proof} If $X=0$, then $\mathrm{cosupp}_{R}X=\emptyset=\mathrm{supp}_{R}X$, and we are done. Next, assume that $X\neq0$. One has that
$\mathrm{supp}_{R}X=\mathrm{Supp}_{R}X$. Let $\mathrm{min}(\mathrm{supp}_{R}X)=\{\mathfrak{p}_1,\cdots,\mathfrak{p}_s\}$. It follows from \cite[Lemma 3.9]{BIK3} that $X\in\mathrm{Thick}_\mathscr{T}(\{X/\hspace{-0.15cm}/\mathfrak{p}_i\hspace{0.03cm}|\hspace{0.03cm}i=1,\cdots,s\})$. Therefore, $\mathrm{cosupp}_{R}X\subseteq\mathcal{V}(\mathfrak{p}_1)\cup\cdots\cup\mathcal{V}(\mathfrak{p}_s)=\mathrm{supp}_{R}X$, as claimed.
\end{proof}

\begin{thm}\label{lem:5.130}{\it{Let $X$ be an object in $\mathscr{T}$ with $\mathrm{supp}_{R}X=\mathrm{Supp}_{R}X$. If $\mathrm{dim}(\mathrm{supp}_{R}X)<\infty$, then there is an inequality
 \begin{center}$\mathrm{cosupp}_{R}X\subseteq\mathrm{supp}_{R}X$.\end{center}}}
\end{thm}
\begin{proof} Assume that $X\neq0$ and
$\mathrm{supp}_{R}X=\mathrm{Supp}_{R}X=\mathcal{V}$. Set $\mathrm{dim}(\mathrm{supp}_{R}X)=n$.
If $n=0$, then $\mathrm{cosupp}_{R}X\subseteq\mathcal{V}$ by \cite[Theorem 4.13]{BIK2}. For $n>0$ set $\mathcal{V}'=\mathcal{V}\backslash\mathrm{min}\hspace{0.05cm}\mathcal{V}$. Then $\mathcal{V}'$ is specialization closed. Since $\mathrm{dim}\hspace{0.05cm}\mathcal{V}'=n-1$, the induction hypothesis yields that $\mathrm{cosupp}_{R}\Gamma_{\mathcal{V}'}X\subseteq\mathcal{V}$. On the other hand, $\mathrm{supp}_R(L_{\mathcal{V}'}X)=\mathrm{min}\hspace{0.05cm}\mathcal{V}$ is discrete and hence $\mathrm{cosupp}_R(L_{\mathcal{V}'}X)=\mathrm{min}\hspace{0.05cm}\mathcal{V}$. Consequently, $\mathrm{cosupp}_{R}X\subseteq\mathcal{V}$, as claimed.
\end{proof}

\begin{rem}\label{lem:0.0}{\rm (1) If $\mathscr{T}$ is generated by finite compact objects $\{C_1,\cdots,C_n\}$, then for any cohomologically finite object $X$ in $\mathscr{T}$,
\begin{center}$\begin{aligned}\mathrm{supp}_{R}X
&=\bigcup_{i=1}^n\mathrm{supp}_{R}\mathrm{H}^\ast_{C_i}(X)\\
&=\bigcup_{i=1}^n\mathcal{V}(\mathrm{ann}_R\mathrm{H}^\ast_{C_i}(X))\\
&=\mathcal{V}(\bigcap_{i=1}^n\mathrm{ann}_R\mathrm{H}^\ast_{C_i}(X)).\end{aligned}$\end{center}
Thus $\mathrm{min}(\mathrm{supp}_{R}X)<\infty$ holds. Set $\mathfrak{a}=\bigcap_{i=1}^n\mathrm{ann}_R\mathrm{H}^\ast_{C_i}(X)$. One also has that $\mathrm{Thick}_\mathscr{T}(X)=\mathrm{Thick}_\mathscr{T}(X/\hspace{-0.15cm}/\mathfrak{a})$.

(2) The inclusion in the preceding theorems can be strict. For example, let $A$ be a commutative noetherian ring and $\mathscr{T}=\mathrm{D}(A)$. If $(A,\mathfrak{m})$ is a complete local ring, then $\mathrm{cosupp}_AA=\{\mathfrak{m}\}\subsetneq\mathrm{Spec}A=\mathrm{supp}_AA$.

(3) Let $A$ be a commutative noetherian ring and $\mathscr{T}=\mathrm{D}(A)$. The cohomologically finite objects in $\mathrm{D}(A)$ are cohomologically bounded complexes
$X$ such that each cohomology module $\mathrm{H}^i(X)$ is finitely generated. In this case, Theorem \ref{lem:5.13} is exactly \cite[Theorem 6.7]{WW}.

(4) Assume that $(A,\mathfrak{m})$ is a commutative local noetherian ring and not artinian and set $I=I(\mathfrak{m})$. Then $\mathrm{supp}_AI=\{\mathfrak{m}\}\subsetneq\mathrm{Spec}A=\mathrm{cosupp}_AI$. Therefore, the assumption that $X$ is cohomologically finite in $\mathscr{T}$ in Theorem \ref{lem:5.13} is essential.

(5) Let $A$ be a commutative noetherian ring and $\mathscr{T}=\mathrm{D}(A)$. Let $X$ be a complex in $\mathrm{D}(A)$ with each $\mathrm{H}^i(X)$ finite length. Then each $\mathfrak{p}\in\mathrm{Supp}_A\mathrm{H}^i(X)$ is a maximal ideal of $A$ by \cite[P.60 Exercise 5]{EJ}. Thus $\mathrm{supp}_AX=\mathrm{Supp}_AX$ is discrete.

(6) If the ring $R$ has finite Krull dimension, then $\mathrm{dim}(\mathrm{supp}_{R}X)<\infty$ for any $X$ in  $\mathscr{T}$.}
\end{rem}

Following \cite{BIK1}, we say that $\mathscr{T}$ is stratified by $R$ if the following conditions hold

(S1) The local-global principle holds for localizing subcategories of $\mathscr{T}$, i.e., for each object $X$
in $\mathscr{T}$, there is an equality
$\mathrm{Loc}_\mathscr{T}(X)=\mathrm{Loc}_\mathscr{T}(\{\Gamma_\mathfrak{p}X\hspace{0.03cm}|\hspace{0.03cm}\mathfrak{p}\in\mathrm{Spec}R\})$.

(S2) For each $\mathfrak{p}\in\mathrm{Spec}R$ the localizing subcategory  $\mathscr{T}_{\{\mathfrak{p}\}}$ has no proper non-zero
localizing subcategories.

The following result was proved by Benson, Iyengar and Krause when $\mathscr{T}$ is noetherian and stratified by $R$ (see \cite[Corollary 5.3]{BIK}).

\begin{prop}\label{lem:4.0}{\it{If $\mathscr{T}$ is stratified by $R$, then for each $C\in\mathscr{T}^c$ and each cohomologically finite object $X$, there is an equality
\begin{center}$\mathrm{supp}_R\mathrm{Hom}^\ast_\mathscr{T}(C,X)=\mathrm{supp}_RC\cap\mathrm{supp}_RX$.\end{center}In particular, $\mathrm{Hom}^\ast_\mathscr{T}(C,X)=0$ if and only if $\mathrm{supp}_RC\cap\mathrm{supp}_RX=\emptyset$.}}
\end{prop}
\begin{proof} Note that $\mathrm{supp}_R\mathrm{Hom}^\ast_\mathscr{T}(C,X)=\mathrm{Supp}_R\mathrm{Hom}^\ast_\mathscr{T}(C,X)$. Let $\mathfrak{p}\in\mathrm{Spec}R$. Suppose $\mathrm{Hom}^\ast_\mathscr{T}(C,X)_\mathfrak{p}\neq0$. Then $\mathrm{Hom}^\ast_\mathscr{T}(C,X_\mathfrak{p})\neq0$, and so $\mathrm{Hom}^\ast_\mathscr{T}(C(\mathfrak{p}),X_\mathfrak{p})\cong\mathrm{Hom}^\ast_\mathscr{T}(C/\hspace{-0.15cm}/\mathfrak{p},X_\mathfrak{p})\neq0$ by \cite[Lemma 5.11(3)]{BIK}. Thus $C(\mathfrak{p})\neq0$ and $X_\mathfrak{p}\neq0$, so $\mathfrak{p}\in\mathrm{supp}_RC\cap\mathrm{supp}_RX$.
Now suppose $\mathrm{Hom}^\ast_\mathscr{T}(C,X)_\mathfrak{p}=0$. Then $\mathrm{Hom}^\ast_\mathscr{T}(C_\mathfrak{p},X_\mathfrak{p})=0$. Note that $\Gamma_\mathfrak{p}C$ is in $\mathrm{Loc}_\mathscr{T}(C_\mathfrak{p})$, one gets $\mathrm{Hom}^\ast_\mathscr{T}(\Gamma_\mathfrak{p}C,\Gamma_\mathfrak{p}X)\cong\mathrm{Hom}^\ast_\mathscr{T}(\Gamma_\mathfrak{p}C,X_\mathfrak{p})=0$. Thus one of $\Gamma_\mathfrak{p}C$ or $\Gamma_\mathfrak{p}X$ is zero since $\mathscr{T}$ is stratifies by $R$. This shows the equality we seek.
\end{proof}

\begin{cor}\label{lem:4.10}{\it{If $\mathscr{T}$ is stratified by $R$, then for each $C\in\mathscr{T}^c$ and each cohomologically finite object $X$, there is an equality
\begin{center}$\mathrm{cosupp}_R\mathrm{Hom}^\ast_\mathscr{T}(X,T_{C}(I(\mathfrak{p})))=\mathrm{supp}_RX\cap\mathrm{cosupp}_RT_{C}(I(\mathfrak{p}))$.\end{center}}}
\end{cor}
\begin{proof} It follows from the definition of the object $T_{C}(I(\mathfrak{p})$ and Proposition \ref{lem:4.0} that \begin{center}$\begin{aligned}\mathrm{cosupp}_R\mathrm{Hom}^\ast_\mathscr{T}(X,T_{C}(I(\mathfrak{p})))
&=\mathrm{supp}_R\mathrm{Hom}^\ast_\mathscr{T}(C,X)\cap\mathrm{cosupp}_RI(\mathfrak{p})\\
&=\mathrm{supp}_RX\cap\mathrm{supp}_RC\cap\mathrm{cosupp}_RI(\mathfrak{p}).\end{aligned}$\end{center} Hence \cite[Proposition 5.4]{BIK2} yields the desired equality.
\end{proof}

Following \cite{BIK1}, we say that $\mathscr{T}$ is costratified by $R$ if the following conditions hold

(C1) The local-global principle holds for colocalizing subcategories of $\mathscr{T}$, i.e., for each object $X$
in $\mathscr{T}$, there is an equality
$\mathrm{Coloc}_\mathscr{T}(X)=\mathrm{Coloc}_\mathscr{T}(\{\Lambda^\mathfrak{p}X\hspace{0.03cm}|\hspace{0.03cm}\mathfrak{p}\in\mathrm{Spec}R\})$.

(C2) For each $\mathfrak{p}\in\mathrm{Spec}R$ the colocalizing subcategory  $\mathscr{T}^{\{\mathfrak{p}\}}$ contains no proper non-zero
colocalizing subcategories.

\begin{lem}\label{lem:2.7}{\it{Suppose $\mathscr{G}$ is a set of compact generators for $\mathscr{T}$. If $\mathscr{T}$ is costratified by $R$, then for each specialization closed subset $\mathcal{V}\subseteq\mathrm{Spec}R$, there are equalities
\begin{center}$\mathscr{T}^\mathcal{V}=\mathrm{Coloc}_\mathscr{T}(C(\mathfrak{p})\hspace{0.03cm}|\hspace{0.03cm}C\in\mathscr{G},\ \mathfrak{p}\in\mathcal{V})=\mathrm{Coloc}_\mathscr{T}(T_{C/\hspace{-0.1cm}/\mathfrak{p}}(I(\mathfrak{p}))\hspace{0.03cm}|\hspace{0.03cm}C\in\mathscr{G},\ \mathfrak{p}\in\mathcal{V})$.\end{center}}}
\end{lem}
\begin{proof} Let $C\in\mathscr{G}$. It follows from \cite[Corollary 4.9]{BIK2} that $\mathrm{cosupp}_RC(\mathfrak{p})\subseteq\mathcal{U}(\mathfrak{p})$. Also $\mathrm{cosupp}_RC(\mathfrak{p})=\mathrm{cosupp}_RC_\mathfrak{p}/\hspace{-0.15cm}/\mathfrak{p}\subseteq\mathcal{V}(\mathfrak{p})$ by \cite[Lemma 4.12]{BIK2}. But $\mathfrak{p}\in\mathrm{supp}_RC$ if and only if $C(\mathfrak{p})\neq0$ if and only if $\mathfrak{p}\in\mathrm{cosupp}_RC(\mathfrak{p})$, so
$\mathrm{cosupp}_RC(\mathfrak{p})=\mathrm{supp}_RC\cap\{\mathfrak{p}\}$. Also $\mathrm{cosupp}_RT_{C/\hspace{-0.1cm}/\mathfrak{p}}(I(\mathfrak{p}))=\mathrm{supp}_RC\cap\{\mathfrak{p}\}$ by \cite[Theorem 5.4]{BIK2}. Therefore, one has the desired equalities by \cite[Remark 5.7]{BIK1}.
\end{proof}

\begin{prop}\label{lem:2.9}{\it{Suppose that $\mathscr{G}$ is a set of compact generators for $\mathscr{T}$. If $\mathscr{T}$ is costratified by $R$, then for each specialization closed subset $\mathcal{V}\subseteq\mathrm{Spec}R$, \begin{center}$(\mathscr{T}^\mathcal{V})^\bot=\{X\in\mathscr{T}\hspace{0.03cm}|\hspace{0.03cm}\mathrm{cosupp}_RX\subseteq\mathrm{Spec}R\backslash\mathcal{V}\}$.\end{center}}}
\end{prop}
\begin{proof} By Proposition \ref{lem:2.13}, one has that
\begin{center}$\mathrm{cosupp}_RX\subseteq\mathrm{Spec}R\backslash\mathcal{V}
\Longleftrightarrow\mathrm{Hom}^\ast_\mathscr{T}(C(\mathfrak{p}),X)=0,\ \forall\ C\in\mathscr{G},\ \forall\ \mathfrak{p}\in\mathcal{V}$.\end{center}Hence Lemma \ref{lem:2.7} shows our claim.
\end{proof}

\bigskip
\section{\bf Big cosupport}
In this section, we introduce the notion of big cosupport for an object, and develop systematically a theory of big cosupport in order to make it viable tool.

\begin{df}\label{lem:4.1}{\rm Let $X$ be an object of $\mathscr{T}$. We define the big cosupport of $X$, denoted by $\mathrm{Cosupp}_RX$,
to be the set
\begin{center}$\mathrm{Cosupp}_RX=\{\mathfrak{p}\in\mathrm{Spec}R\hspace{0.03cm}|\hspace{0.03cm}V^{\mathcal{Z}(\mathfrak{p})}X\neq0\}$.\end{center}}
\end{df}

\begin{rem}\label{lem:4.3}{\rm $\mathrm{(1)}$ For any object $X$ in $\mathscr{T}$, one has that $\mathrm{Cosupp}_{R}X=\mathrm{Cosupp}_{R}\Sigma X$ and $\mathrm{cosupp}_RX\subseteq\mathrm{Cosupp}_RX$.

$\mathrm{(2)}$ For any exact triangle $X\rightarrow Y\rightarrow Z\rightsquigarrow$ in $\mathscr{T}$, we have
\begin{center}$\mathrm{Cosupp}_{R}Y\subseteq\mathrm{Cosupp}_{R}X\cup\mathrm{Cosupp}_{R}Z$.\end{center}

$\mathrm{(3)}$ For any object $X$ in $\mathscr{T}$, the subset $\mathrm{Cosupp}_RX$ is specialization closed.

$\mathrm{(4)}$ $\mathrm{Cosupp}_R(\Lambda^\mathcal{V}X)=\mathrm{Cosupp}_RX\cap\mathcal{V}$ by \cite[(4.2)]{BIK2}, and so
 $\mathrm{Cosupp}_RX=\mathrm{Cosupp}_R(\Lambda^\mathcal{V}X)\cup\mathrm{Cosupp}_R(V^\mathcal{V}X)$ by the exact triangle $V^\mathcal{V}X\rightarrow X\rightarrow\Lambda^\mathcal{V}X\rightsquigarrow$, where $\mathcal{V}$ is a specialization closed subset of $\mathrm{Spec}R$.}
\end{rem}

The following result provides a computation of big cosupport.

\begin{thm}\label{lem:4.2}{\it{For each each object $X$ in $\mathscr{T}$, one has that
\begin{center}$\mathrm{Cosupp}_{R}X=\bigcup_{C\in\mathscr{T}^c,\ \mathfrak{p}\in\mathrm{Spec}R}\mathrm{Cosupp}_R\mathrm{Hom}^\ast_\mathscr{T}(V^{\mathcal{Z}(\mathfrak{p})}X,T_{C}(I(\mathfrak{p})))$.
\end{center}In particular, $X=0$ if and only if $\mathrm{Cosupp}_RX=\emptyset$.}}
\end{thm}
\begin{proof} Let $\mathfrak{p}\in\mathrm{Spec}R$ be such that $V^{\mathcal{Z}(\mathfrak{p})}X\neq0$. By Theorem \ref{lem:4.40}, there is a compact object $C\in\mathscr{T}^c$ such that $\mathrm{cosupp}_R\mathrm{Hom}^\ast_\mathscr{T}(V^{\mathcal{Z}(\mathfrak{p})}X,T_{C}(I(\mathfrak{p})))\neq0$. But
\begin{center}$\mathrm{Hom}^\ast_R(R_\mathfrak{p},\mathrm{Hom}^\ast_\mathscr{T}(V^{\mathcal{Z}(\mathfrak{p})}X,T_{C}(I(\mathfrak{p}))))
\cong\mathrm{Hom}^\ast_\mathscr{T}(V^{\mathcal{Z}(\mathfrak{p})}X,T_{C}(I(\mathfrak{p})))$,\end{center}so $\mathfrak{p}\in\mathrm{Cosupp}_R\mathrm{Hom}^\ast_\mathscr{T}(V^{\mathcal{Z}(\mathfrak{p})}X,T_{C}(I(\mathfrak{p})))$.
Conversely, given a point $\mathfrak{q}$ in $\mathrm{Spec}R$ such that $\mathrm{Hom}^\ast_R(R_\mathfrak{q},\mathrm{Hom}^\ast_\mathscr{T}(V^{\mathcal{Z}(\mathfrak{p})}X,T_{C}(I(\mathfrak{p}))))\neq0$ for some $C\in\mathscr{T}^c$ and $\mathfrak{p}\in\mathrm{Spec}R$. Assume to the contrary that $V^{\mathcal{Z}(\mathfrak{q})}X=0$. So $\mathrm{cosupp}_{R}X\subseteq\mathcal{Z}(\mathfrak{q})$ by \cite[Corollary 4.8]{BIK2}. Therefore, $\mathrm{cosupp}_R\mathrm{Hom}^\ast_\mathscr{T}(V^{\mathcal{Z}(\mathfrak{u})}X,T_{D}(I(\mathfrak{u})))\subseteq\mathcal{Z}(\mathfrak{q})$ for any $D\in\mathscr{T}^c$ and $\mathfrak{u}\in\mathrm{Spec}R$ by Corollary \ref{lem:1.10}, that is to say, $\mathrm{Hom}_R(R_\mathfrak{q},\mathrm{Hom}^\ast_\mathscr{T}(V^{\mathcal{Z}(\mathfrak{u})}X,T_{D}(I(\mathfrak{u}))))=0$ for all $D\in\mathscr{T}^c$ and $\mathfrak{u}\in\mathrm{Spec}R$, which contradicts our assumption.
\end{proof}

\begin{prop}\label{lem:4.11}{\it{Let $X$ be an object of $\mathscr{T}$. The sets $\mathrm{cosupp}_RX$ and $\mathrm{Cosupp}_RX$ have the same minimal elements with respect to containment, i.e. $\mathrm{min}(\mathrm{cosupp}_RX)=\mathrm{min}(\mathrm{Cosupp}_RX)$.}}
\end{prop}
\begin{proof} For the containment $\mathrm{min}(\mathrm{cosupp}_RX)\supseteq\mathrm{min}(\mathrm{Cosupp}_RX)$, fix $\mathfrak{p}\in\mathrm{min}(\mathrm{Cosupp}_R(X))$. Then $\mathrm{cosupp}_R(V^{\mathcal{Z}(\mathfrak{p})}X)\neq\emptyset$, this is to say,
there is $\mathfrak{q}\in\mathrm{Spec}R$ such that $\Lambda^\mathfrak{q}(V^{\mathcal{Z}(\mathfrak{p})}X)\neq0$. But $\mathfrak{q}\subseteq \mathfrak{p}$ by \cite[Corollary 4.9]{BIK2}, so
$\Lambda^\mathfrak{q}X\cong\Lambda^\mathfrak{q}(V^{\mathcal{Z}(\mathfrak{p})}X)\neq0$, which implies that $\mathfrak{q}\in\mathrm{cosupp}_RX\subseteq\mathrm{Cosupp}_RX$. Thus the minimality of $\mathfrak{p}$ in $\mathrm{Cosupp}_RX$
 implies that $\mathfrak{p}=\mathfrak{q}\in\mathrm{cosupp}_RX$. From the containment $\mathrm{cosupp}_RX\subseteq\mathrm{Cosupp}_RX$, the fact that $\mathfrak{p}$ is minimal in $\mathrm{Cosupp}_RX$
implies that it is also minimal in $\mathrm{cosupp}_RX$.

For the reverse containment, let $\mathfrak{p}\in\mathrm{min}(\mathrm{cosupp}_RX)\subseteq\mathrm{Cosupp}_RX$. Suppose that $\mathfrak{p}$ is not
minimal in $\mathrm{Cosupp}_RX$, so there is a prime $\mathfrak{q}\in\mathrm{Cosupp}_RX$ such that $\mathfrak{q}\subsetneq\mathfrak{p}$, that is to say,
$V^{\mathcal{Z}(\mathfrak{q})}X\neq0$, so there is a prime $\mathfrak{u}\in\mathrm{Spec}R$ such that $\mathfrak{u}\subseteq \mathfrak{q}$ and $\mathfrak{u}\in \mathrm{cosupp}_{R}(V^{\mathcal{Z}(\mathfrak{q})}X)$. As in the previous
paragraph, this implies that $\mathfrak{u}\in\mathrm{cosupp}_RX$, so the minimality of $\mathfrak{p}$ implies that $\mathfrak{p}=\mathfrak{u}\subseteq\mathfrak{q}\subsetneq\mathfrak{p}$, which contradicts the assumption.
\end{proof}

\begin{cor}\label{lem:4.12}{\it{Let $X$ be an object of $\mathscr{T}$ and $\mathfrak{a}$ a homogeneous ideal of $R$.
One has that $\mathrm{Cosupp}_RX\subseteq\mathcal{V}(\mathfrak{a})$ if and only if $\mathrm{cosupp}_RX\subseteq\mathcal{V}(\mathfrak{a})$.}}
\end{cor}
\begin{proof}  The forward implication is by the containment $\mathrm{cosupp}_RX\subseteq\mathrm{Cosupp}_RX$. For the converse,
assume that $\mathrm{cosupp}_RX\subseteq\mathcal{V}(\mathfrak{a})$, and let $\mathfrak{p}\in\mathrm{Cosupp}_RX$. It follows that $\mathfrak{p}$ is contained in a minimal element
$\mathfrak{q}$ of $\mathrm{Cosupp}_RX$, which is in $\mathrm{cosupp}_RX$ by Proposition \ref{lem:4.11}. In other words, we have $\mathfrak{a}\subseteq \mathfrak{q}\subseteq\mathfrak{p}$,
so $\mathfrak{p}\in\mathcal{V}(\mathfrak{a})$.
\end{proof}

\bigskip \centerline {\bf ACKNOWLEDGEMENTS}
\bigskip The paper began during a visit to the University of Utah. I am very grateful to Professor Srikanth Iyengar for a lot of helpful discussions and reading many versions of this paper. This research was partially supported by National Natural Science Foundation of China (11761060).

\end{document}